\documentclass{article}

\usepackage{cite}
\usepackage{amssymb}
\usepackage{amsmath}
\usepackage{amsfonts}
\usepackage{amsthm}
\usepackage{enumerate}
\usepackage{colonequals}
\usepackage{mathrsfs} 
\usepackage{color}
\usepackage{enumitem}
\usepackage{mathtools}
\usepackage[all,cmtip]{xy}
\usepackage{tikz-cd} 
\usepackage{multicol}
\usepackage{float}
\usepackage{placeins}
\usepackage{accents}
\usepackage{titlesec}

\usetikzlibrary{decorations.markings}

\usepackage{url}

\usepackage{euscript}
\usepackage{wasysym}

\newtheorem{theorem}{Theorem}[section]

\newtheorem{proposition}{Proposition}[section]

\newtheorem*{theorem*}{Theorem}
\newtheorem{corollary}{Corollary}[section]
\newtheorem{definition}{Definition}[section]

\newtheorem{lemma}{Lemma}[section]

\newtheorem{remark}{Remark}[section]
\newtheorem{conjecture}{Conjecture}[section]

\newcommand{\eusc}[1]{\EuScript{#1}}

\numberwithin{equation}{section}

\renewcommand{\sf}[1]{\mathsf{#1}}

\allowdisplaybreaks[2]

\newcommand{\deee}{\hspace{2 pt} \mathrm{d}}

\newcommand{\ov}[1]{\overline{#1}}

\newcommand{\mrm}[1]{\mathrm{#1}}
\newcommand{\nml}{\left \vert \left \vert}
\newcommand{\nmr}{\right \vert \right \vert}

\newcommand{\pmp}{p.m.p.\,}
\newcommand{\qqquad}{\qquad \qquad}

\newcommand{\parr}{\hspace{-2.5 pt} \parallel \hspace{-2.5 pt}}

\DeclareMathOperator{\uotimes}{\underline{\otimes}}
\DeclareMathOperator{\ootimes}{\ov{\otimes}}

\DeclareMathOperator{\sqq}{\square}

  \usepackage{stackengine,scalerel}

\allowdisplaybreaks

\begin{document}

\title{Synergodic actions of product groups}
\author{Peter Burton}

\maketitle

\begin{abstract} This article studies a structural aspect of measure-preserving actions of products of countable discrete groups, involving a so-called `synergodic decomposition' in terms of the ergodic components of the actions of the two factor groups. We show that this construction provides a canonical way to detect whether the action is built as a product of actions on independent measure spaces, and we use it to prove a result about convergence of ergodic averages on product groups which are highly imbalanced between the factors. Defining an action to be synergodic if it is isomorphic to its synergodic decomposition, we show that if a countable group $G$ is amenable then every action of $G \times G$ can be approximated by synergodic actions and that this statement fails if $G$ is a nonabelian free group. The last result relies on the refutation of Connes' embedding conjecture.\end{abstract}

\tableofcontents

\section{Introduction}

\subsection{Preliminaries on ergodic theory} \label{subsec.ergprelim}

\subsubsection{Measure-preserving transformations}

We begin by recalling some aspects of the theory of measure preserving transformations. References for this material are Chapter 17 in \cite{MR1321597} and Section 1 of Chapter 1 in \cite{MR2583950}. If $(X,\mu)$ is a standard probability space, we will use the notation $\eusc{M}_X$ for the $\sigma$-algebra of all measurable subsets of $X$, and we define a \textbf{measurable }$\sigma$\textbf{-algebra on }$X$ to be a sub-$\sigma$-algebra of $\eusc{M}_X$. We also write $\eusc{N}_X$ for the trivial $\sigma$-algebra on $X$ given by: \[ \eusc{N}_X = \bigl\{D \in \eusc{M}_X: \mu(D) \in \{0,1\}  \bigr\} \] We let $\mathrm{Aut}(X,\mu)$ denote the group of all probability-preserving transformations of $(X,\mu)$ and if $\mathsf{T} \in \mrm{Aut}(X,\mu)$ and $x \in X$ we will typically simplify notation by writing $\mathsf{T}x$ for the element $\mathsf{T}(x)$ of $X$. Similarly, for $D \in \eusc{M}_X$ we will write $\mathsf{T}D$ for the element $\{\mathsf{T}x: x \in D\}$ of $\eusc{M}_X$. We adopt the convention that elements $\sf{S}$ and $\sf{T}$ of $\mrm{Aut}(X,\mu)$ are identified if $\sf{S}x = \sf{T}x$ for $\mu$-almost every $x \in X$, or equivalently if $\mu(\sf{S}D \triangle \sf{T}D) =0$ for all $D \in \eusc{M}_X$. We also define $D$ to be $\sf{T}$\textbf{-invariant} if $\mu(\sf{T}D \triangle D) = 0$.\\ 
\\
We endow $\mrm{Aut}(X,\mu)$ with the \textbf{weak topology}, which is defined by the stipulation that a sequence $(\sf{T}_n)_{n \in \mathbb{N}}$ of elements of $\mrm{Aut}(X,\mu)$ converges to $\sf{T} \in \mrm{Aut}(X,\mu)$ if and only if we have \[ \lim_{n \to \infty} \mu(\sf{T}_n D \triangle \sf{T}D) = 0 \] for all $D \in \eusc{M}_X$. With this topology, $\mrm{Aut}(X,\mu)$ is a Polish space. While there are other plausible topologies on $\mrm{Aut}(X,\mu)$, we will have no cause to consider them in this article and so any topological assertion about $\mrm{Aut}(X,\mu)$ should be interpreted with respect to the weak topology. 

\subsubsection{Actions of countable discrete groups}

We now recall some aspects of the theory of measurable group actions, and we refer the reader to Chapter 10 of \cite{MR2583950} for more information on these topics. Throughout the article, we will understand all countable groups as discrete.  If $G$ is a countable group, we will refer to a homomorphism from $G$ to $\mrm{Aut}(X,\mu)$ as a \textbf{probability-measure preserving (\pmp \!) action of }$G$\textbf{ on }$(X,\mu)$. We write $\mrm{Act}(G,X,\mu)$ for the set of all \pmp actions of $G$ on $(X,\mu)$ and endow $\mrm{Act}(G,X,\mu)$ with the Polish topology it inherits as a closed subspace of the product space $\mathrm{Aut}(X,\mu)^G$. Given an element $\sf{A}$ of $\mrm{Act}(G,X,\mu)$ and $g \in G$, we will typically write $\sf{A}^g$ for the element $\sf{A}(g)$ of $\mrm{Aut}(X,\mu)$.  \\
\\
We define a set $D \in \eusc{M}_X$ to be $\sf{A}$\textbf{-invariant} if $D$ is $\mathsf{A}^g$-invariant for all $g \in G$. Writing $\eusc{E}_\sf{A}$ for the $\sigma$-algebra of $\sf{A}$-invariant subsets of $X$, we define an action $\sf{A} \in \mrm{Act}(G,X,\mu)$ to be \textbf{ergodic} if $\eusc{E}_\sf{A} = \eusc{N}_X$. We also write $\mathrm{Erg}(G,X,\mu)$ for the subset of $\mrm{Act}(G,X,\mu)$ consisting of ergodic actions and write $\mathsf{I}_X^G$ for the trivial action of $G$ on $(X,\mu)$. If $(Y,\nu)$ is another standard probability space, for $\sf{A} \in \mrm{Act}(G,X,\mu)$ and $\sf{B} \in \mrm{Act}(G,Y,\nu)$, we will consider the \textbf{diagonal action} $\sf{A} \parr \sf{B} \in \mrm{Act}(G,X \times Y,\mu \times \nu)$ given by setting $(\sf{A} \parr \sf{B})^g(x,y) = (\sf{A}^gx,\sf{B}^g y)$ for $(x,y) \in X \times Y$ and $g \in G$. 

\begin{remark} \label{rem.overload} As we will be considering multiple notions of `product actions', all of which could plausibly be denoted by $\sf{A} \times \sf{B}$, we use this nonstandard $\parallel$ notation for diagonal actions to avoid overloading the $\times$ symbol. \end{remark}

\subsubsection{Factor maps and isomorphisms}

If $(X,\mu)$ and $(Y,\nu)$ are standard probability spaces, $\Phi:X \to Y$ is a measurable map and $\eusc{S}$ is a measurable $\sigma$-algebra on $Y$, we define the \textbf{preimage} $\Phi_{\uparrow}[\eusc{S}]$ to be the measurable $\sigma$-algebra $\Phi_{\uparrow}[\eusc{S}]$ on $X$ given by: \[ \Phi_{\uparrow}[\eusc{S}] = \{ \Phi^{-1}(D):D \in \eusc{S} \} \] We also write $\Phi_{\#}\mu$ for the \textbf{pushfoward measure} on $Y$ given by setting $[\Phi_{\#}\mu](D) = \mu(\Phi^{-1}(D))$ for $D \in \eusc{M}_Y$. Now, let $G$ be a countable group, let $\sf{A} \in \mrm{Act}(G,X,\mu)$ and let $\sf{B} \in \mrm{Act}(G,Y,\nu)$. We define a \textbf{factor map from }$\sf{A}$\textbf{ onto }$\sf{B}$ to be a measurable surjection $\Phi:X \twoheadrightarrow Y$ such that $\Phi_{\#}\mu = \nu$ and such that $\Phi \circ \sf{A}^g = \sf{B}^g \circ \Phi$ for all $g \in G$. We also define a factor map $\Phi$ from $\sf{A}$ to $\sf{B}$ be a \textbf{isomorphism} if it is a bijection between $X$ and $Y$. \\
\\
Suppose now that $(Z,\omega)$ is another standard probability space and $\mathsf{C} \in \mrm{Act}(G,Z,\omega)$ is an action which is a factor of $\mathsf{A}$ via the map $\Omega$. We define the actions $\mathsf{B}$ and $\mathsf{C}$ to be \textbf{isomorphic as factors of }$\mathsf{A}$ if there exists an isomorphism $\Psi:Y \to Z$ between $\mathsf{B}$ and $\mathsf{C}$ such that the following diagram commutes.
\[ \begin{tikzcd}& X   \arrow[ddl,"\Phi" above left,two heads]  \arrow[ddr, "\Omega" above right,two heads]  & \\ \\ Y \arrow[rr, "\Psi" ,hook, two heads] & & Z \end{tikzcd} \]

 Note that if $X = Y$ and $\sf{T} \in \mrm{Aut}(X,\mu)$ then $\sf{T}_{\uparrow}[\eusc{S}]$ is a sub-$\sigma$-algebra of $\eusc{M}_X$. Thus for $\sf{A} \in \mrm{Act}(G,X,\mu)$ we may define a sub-$\sigma$-algebra $\eusc{S}$ of $\eusc{M}_X$ to be $\sf{A}$\textbf{-stable} if $\sf{A}^g_{\uparrow}[\eusc{S}] = \eusc{S}$ for all $g \in G$.  A basic construction in ergodic theory is the following `fundamental theorem of factor maps', which appears as Theorem 2.15 in \cite{MR1958753}.

\begin{theorem} \label{thm.factor} Let $G$ be a countable group, let $(X,\mu)$ be a standard probability space and let $\mathsf{A} \in \mrm{Act}(G,X,\mu)$. Then for any $\sf{A}$-stable measurable $\sigma$-algebra $\eusc{S}$ on $X$ there exists a standard probability space $(Y,\nu)$, an action $\mathsf{B} \in \mrm{Act}(G,Y,\nu)$ and a factor map $\Phi:X \twoheadrightarrow Y$ from $\mathsf{A}$ onto $\mathsf{B}$ such that $\eusc{S} = \Phi_{\uparrow}[\eusc{M}_Y]$.  \end{theorem} We will refer to any factor as above a \textbf{realization} of the $\sf{A}$-stable $\sigma$-algebra $\eusc{S}$. Theorem 2.15 in \cite{MR1958753} also asserts that any two realizations of $\eusc{S}$ are isomorphic as factors of $\mathsf{A}$, and so we will typically make no distinction between such realizations. It is straightforward to verify that the $\sf{A}$-invariant $\sigma$-algebra $\eusc{E}_\sf{A}$ is always $\sf{A}$-stable. In keeping with standard terminology, we will refer to the factor of $\sf{A}$ realizing $\eusc{E}_\sf{A}$ as the \textbf{ergodic decomposition of }$\sf{A}$. We will use the notation $(E_\mathsf{A},\eta_\mathsf{A})$ for the associated standard probability space, which we refer to as the \textbf{space of ergodic components of }$\mathsf{A}$.

\subsubsection{Actions of product groups}

For countable groups $G$ and $H$, an element of $\mrm{Act}(G \times H,X,\mu)$ may be identified with a pair of elements $\sf{A} \in \mrm{Act}(G,X,\mu)$ and $\sf{B} \in \mrm{Act}(H,X,\mu)$ which commute in the sense that $\sf{A}^g \sf{B}^h = \sf{B}^h\sf{A}^g$ for all $(g,h) \in G \times H$. As our primary interest in this article will be studying actions of product groups, we will typically use this approach to discussing actions of $G \times H$. Thus a statement such as `let $(\sf{A},\sf{B}) \in \mrm{Act}(G \times H,X,\mu)$' should be understood as an assertion that $\sf{A}$ and $\sf{B}$ are respectively elements of $\mathrm{Act}(G,X,\mu)$ and $\mrm{Act}(H,X,\nu)$, and that these actions commute. We will also simplify notation by writing $\eusc{E}_{\sf{A},\sf{B}}$ instead of $\eusc{E}_{(\sf{A},\sf{B})}$ for the measurable $\sigma$-algebra on $X$ consisting of sets which are invariant with respect to both $\sf{A}$ and $\sf{B}$.\\
\\
Now, consider standard probability spaces $(X,\mu)$ and $(Y,\nu)$ and actions $\sf{A} \in \mrm{Act}(G,X,\mu)$ and $\sf{B} \in \mrm{Act}(H,Y,\nu)$. We define the \textbf{local product of }$\sf{A}$\textbf{ and }$\sf{B}$ to be the element $\sf{A} \sqq \sf{B}$ of $\mrm{Act}(G \times H,X \times Y,\mu \times \nu)$ given by letting $(\mathsf{A}\sqq \mathsf{B})^{g,h}(x,y) = (\mathsf{A}^gx,\mathsf{B}^hy)$ for $(g,h) \in G \times H$ and $(x,y) \in X \times Y$. (Here Remark \ref{rem.overload} is again relevant.) In this case we use the notation $\mathsf{A}_\square$ for the restriction of $\mathsf{A} \sqq \mathsf{B}$ to $G$ and similarly we use the notation ${}_\square\mathsf{B}$ for the restriction of $\mathsf{A} \sqq \mathsf{B}$ to $H$.

\begin{remark}\label{rem.loc} With the notation of the previous paragraph, the action $\mathsf{A}_\square$ can be identified with the diagonal action $\mathsf{A} \parr \sf{I}_G^Y \in \mrm{Act}(G,X \times Y,\mu \times \nu)$ and ${}_\square\mathsf{B}$ can be identified with the diagonal action $\mathsf{I}_H^X \parr \mathsf{B} \in \mrm{Act}(H,X \times Y,\mu \times \nu)$. Thus the local product may be alternatively expressed as: \[ \mathsf{A} \sqq \mathsf{B}= \bigl(\mathsf{A} \parr \sf{I}_G^Y  , \mathsf{I}_H^X\parr \mathsf{B} \bigr) \] \end{remark}

The purpose of this article is to investigate a certain `synergodic decomposition' which in many cases allows for a canonical representation of an action of a product group as a local product. 

\subsection{Formulation and statement of results}

\subsubsection{The synergodic decomposition and local products}

 If $G$ and $H$ are countable groups and $(\sf{A},\sf{B}) \in \mrm{Act}(G \times H,X,\mu)$, it is straightforward to verify that the $\sigma$-algebra $\eusc{E}_\sf{A}$ is $\sf{B}$-stable. Another way of phrasing this statement is to assert that the action $(\sf{A},\sf{B})$ gives rise to a factor of $\mathsf{B}$ acting on the space of $(E_\mathsf{A},\eta_\mathsf{A}) $ of ergodic components of $\mathsf{A}$. We will denote this factor action by $\mathsf{E}_{\mathsf{B} \curvearrowright \mathsf{A}} \in \mrm{Act}(H,E_\mathsf{A},\eta_\mathsf{A})$. The same reasoning shows that $\eusc{E}_{\sf{B}}$ is $\sf{A}$-stable and so we obtain an analogous action $\mathsf{E}_{\mathsf{A} \curvearrowright \mathsf{B}} \in \mrm{Act}(G,E_\mathsf{B},\eta_\mathsf{B})$. Write $\sigma(\eusc{E}_\mathsf{A},\eusc{E}_\mathsf{B})$ for the measureable $\sigma$-algebra on $X$ generated by $\eusc{E}_{\sf{A}}$ and $\eusc{E}_{\sf{B}}$. Since both $\eusc{E}_{\mathsf{A}}$ and $\eusc{E}_\mathsf{B}$ are $(\mathsf{A},\mathsf{B})$-stable we have that $\sigma(\eusc{E}_{\sf{A}},\eusc{E}_{\sf{B}})$ is again $(\sf{A},\sf{B})$-stable, and so the following definition is justified.

\begin{definition} Let $G$ and $H$ be countable discrete groups, let $(X,\mu)$ be a standard probability space and let $(\sf{A},\sf{B}) \in \mrm{Act}(G \times H,X,\mu)$. We define the \textbf{synergodic decomposition of }$(\sf{A},\sf{B})$ to be the factor of $(\mathsf{A},\mathsf{B})$ realizing the stable $\sigma$-algebra $\sigma(\eusc{E}_{\sf{A}},\eusc{E}_{\sf{B}})$. We also define the action $(\sf{A},\sf{B})$ to be \textbf{synergodic} if it is ergodic and the factor map in its synergodic decomposition is an isomorphism, or equivalently if $\eusc{E}_{\sf{A},\sf{B}} = \eusc{N}_X$ and $\sigma(\eusc{E}_{\sf{A}},\eusc{E}_{\sf{B}}) = \eusc{M}_X$. \end{definition}

Our first result is the following.

\begin{theorem} \label{thm.synerg} Let $G$ and $H$ be countable discrete groups, let $(X,\mu)$ be a standard probability space and let $(\sf{A},\sf{B}) \in \mrm{Erg}(G \times H,X,\mu)$. Then the synergodic decomposition of $(\sf{A},\sf{B})$ is isomorphic as a factor of $(\sf{A},\sf{B})$ to the local product: \[(\mathsf{E}_{\mathsf{A} \curvearrowright \mathsf{B}}) \sqq (\mathsf{E}_{\mathsf{B} \curvearrowright \mathsf{A}}) \in \mrm{Act}(G \times H,E_\mathsf{A} \times E_\mathsf{B}, \eta_\mathsf{A} \times \eta_\mathsf{B}) \]  \end{theorem}

Having found that synergodic decompositions are local products, we also obtain the following assurance that the synergodic decomposition will recover a local product of ergodic actions.

\begin{theorem} \label{thm.locprod} Let $G$ and $H$ be countable groups, let $(X,\mu)$ and $(Y,\nu)$ be standard probability spaces and let $\sf{A} \in \mrm{Erg}(G,X,\mu)$ and $\sf{B} \in \mrm{Erg}(H,Y,\nu)$ be ergodic actions. Then the local product $\sf{A} \sqq \sf{B} \in \mrm{Act}(G \times H,X \times Y,\mu \times \nu)$ is synergodic. Moreover, in this local product we have $\eusc{E}_{\mathsf{A}_\square} = \eusc{N}_X \times \eusc{M}_Y$ and $\eusc{E}_{{}_\square\mathsf{B}} = \eusc{M}_X \times \eusc{N}_Y$.  \end{theorem}

It follows directly from the last statement in Theorem \ref{thm.locprod} that $\mathsf{A}$ and $\mathsf{E}_{\mathsf{A} \curvearrowright \mathsf{B}}$ are isomorphic as actions of $G$, and $\mathsf{B}$ and $\mathsf{E}_{\mathsf{B} \curvearrowright \mathsf{A}}$ are isomorphic as actions of  $H$. We may combine Theorems \ref{thm.synerg} and \ref{thm.locprod} into the following slightly less formal statement.

\begin{corollary} \label{cor.informal} Let $G$ and $H$ be countable discrete groups, let $(X,\mu)$ be a standard probability space and let $(\sf{A},\sf{B}) \in \mrm{Erg}(G \times H,X,\mu)$. Then $(\mathsf{A},\mathsf{B})$ is synergodic if and only if it is isomorphic to the local product $\mathsf{A} \sqq \mathsf{B}$. Moreover, when these equivalent conditions are satisfied the local product structure of $(\mathsf{A},\mathsf{B})$ may be constructed `internally' as $(\mathsf{E}_{\mathsf{A} \curvearrowright \mathsf{B}}) \sqq (\mathsf{E}_{\mathsf{B} \curvearrowright \mathsf{A}})$. \end{corollary}

Corollary \ref{cor.informal} guarantees that we can construct examples of synergodic actions from ergodic actions of the factor groups, and that this is essentially the only way to construct synergodic actions. Thus we are lead to ask for examples of ergodic actions of product groups which fail to be synergodic. Such nonexamples indeed exist for the Cartesian square of any infinite group, as we now explain. \\
\\
Let $G$ be a countable group and let $(X,\mu)$ be any standard probability space. We consider the space $(X^G,\mu^G)$ of $X$-valued sequences on $G$ with the product measure $\mu^G$ on the $\sigma$-algebra generated by cylinder sets. This space admits two commuting \pmp actions of $G$, the \textbf{left and right Bernoulli shifts of }$G$\textbf{ over }$(X,\mu)$ defined respectively by \[ (\mathsf{L}^h \xi)(g) = \xi(h^{-1}g) \qqquad \qquad \qqquad (\mathsf{R}^h \xi)(g) = \xi(gh) \] for $\xi \in X^G$ and $g,h \in G$. These two actions are isomorphic by the involution $\mathsf{J} \in \mrm{Aut}(X^G,\mu^G)$ given by $(\mathsf{J}\xi)(g) = \mathsf{J}(g^{-1})$. More explicitly, we have $\mathsf{J} = \mathsf{J}^{-1}$ and for all $g \in G$ the Bernoulli shifts satisfy $\mathsf{L}^g =  \mathsf{J} \mathsf{R}^g \mathsf{J}$. It is standard that both of the one-sided Bernoulli shifts are ergodic elements of $\mrm{Act}(G,X^G,\mu^G)$ as long as $G$ is infinite.\\
\\
We define the \textbf{two-sided Bernoulli shift over }$(X,\mu)$ to be the action $(\mathsf{L},\mathsf{R}) \in \mrm{Act}(G \times G, X^G,\mu^G)$. Since both $\sf{L}$ and $\sf{R}$ are ergodic, we have $\eusc{E}_\sf{L} = \eusc{E}_\sf{R} = \eusc{N}_{X^G}$ and therefore $\sigma(\eusc{E}_\sf{L},\eusc{E}_\sf{R}) = \eusc{N}_{X^G}$. It follows that the two-sided Bernoulli shift is not synergodic as an action of $G \times G$, as long as the probability space $(X,\mu)$ is nontrivial. Indeed, this line of reasoning shows that a nontrivial action of a product group where the individual factors act ergodically can never be synergodic.
 
\subsubsection{Eccentric ergodic averages}

Let $p \in [1,\infty]$, let $(X,\mu)$ be a standard probability space and let $\sf{A} \in \mrm{Act}(G,X,\mu)$. For a probability measure $\mathfrak{w}$ on $G$ it will be convenient to use the notation $\sf{A}(\mathfrak{w})$ for the operator defined on $\psi \in L^p(X,\mu)$ by: \begin{equation} \label{eq.defop} [\sf{A}(\mathfrak{w})\psi](x) = \sum_{g \in G} \mathfrak{w}(g) \psi(\mathsf{A}^g x)  \end{equation} We note that $\sf{A}(\mathfrak{w})$ contracts the $p$-norm for all $p \in [1,\infty]$. We also let $\mathbb{I}_\sf{A}$ denote the conditional expectation operator onto the $\sf{A}$-invariant $\sigma$-algebra $\eusc{E}_\sf{A}$, so $\mathbb{I}_\sf{A}$ again contracts the $p$-norm for all $p \in [1,\infty]$. 

\begin{definition} Let $G$ be a countable group and let $(\mathfrak{w}_n)_{n \in \mathbb{N}}$ be a sequence of probability measures on $G$. For $p \in [1,\infty]$, we define $(\mathfrak{w}_n)_{n \in \mathbb{N}}$ to be $p$\textbf{-mean ergodic} if for all standard probability spaces $(X,\mu)$, all \pmp actions $\sf{A} \in \mrm{Act}(G,X,\mu)$ and all $\psi \in L^p(X,\mu)$ we have: \[ \lim_{n \to \infty} ||(\mathsf{A}(\mathfrak{w}_n) - \mathbb{I}_{\sf{A}})\psi ||_p = 0 \]  \end{definition} By analyzing limits of ergodic averages over the individual factor groups in terms of the synergodic decomposition, we are able to show the following.

\begin{theorem} \label{thm.crosserg} For any pair of countable groups $G$ and $H$ and any $p \in [1,\infty]$ the following are equivalent. \begin{description} \item[(a)] Both of the sequences $(\mathfrak{w}_n)_{n \in \mathbb{N}}$ and $(\mathfrak{v}_n)_{n \in \mathbb{N}}$ are $p$-mean ergodic. \item[(b)] The sequence $(\mathfrak{w}_n \times \mathfrak{v}_n)_{n \in \mathbb{N}}$ is $p$-mean ergodic.  \end{description} \end{theorem}

Recall that a sequence $(F_n)_{n \in \mathbb{N}}$ of nonempty subsets of a countable group $G$ is defined to be a \textbf{F\"{o}lner sequence} we have \[ \lim_{n \to \infty} \frac{|gF_n \triangle F_n|}{|F_n|} = 0 \] for all $g \in G$. The group $G$ is defined to be \textbf{amenable} if there exists a F\"{o}lner sequence in $G$. In the classical case where $G$ and $H$ are amenable groups and $(\mathfrak{w}_n)_{n \in \mathbb{N}}$ and $(\mathfrak{u}_n)_{n \in \mathbb{N}}$ are uniform averages on F\"{o}lner sequences, Theorem \ref{thm.crosserg} follows from the fact that $(\mathfrak{w}_n \times \mathfrak{u}_n)_{n \in \mathbb{N}}$ is again a sequences of uniform averages on F\"{o}lner sequence in $G \times H$. (See \cite{MR1865397} for the ergodic theorem on amenable groups.)  \\
\\
In the case of nonamenable groups, ergodic theorems typically rely on the averages in question either having some amount of geometric regularity or on being powers of a single Markov operator. We direct the reader to the references \cite{MR2573139}, \cite{MR2186253} or to the French-language text \cite{MR2643350} for surveys of this extensive topic. Theorem \ref{thm.crosserg} provides apparently new instances of ergodic theorems on products on nonamenable groups where the averaging sets are highly `eccentric' between the two factors. \\
\\
For example, let $d \in \mathbb{N}$ and consider the rank-$d$ free group $\mathbb{F}_d$. Theorem 1.1 in \cite{MR3113602} and Theorem 1 in \cite{MR1266737}  describe a variety of examples of $p$-mean ergodic sequences $(\mathfrak{w}_n)_{n \in \mathbb{N}}$ on $\mathbb{F}_d$. From Theorem \ref{thm.crosserg} we find that for \textit{any} function $\tau:\mathbb{N} \to \mathbb{N}$ with $\lim_{n \to \infty} \tau(n) = \infty$ the sequence $(\mathfrak{w}_n \times \mathfrak{w}_{\tau(n)})_{n \in \mathbb{N}}$ is $p$-mean ergodic on $\mathbb{F}_d \times \mathbb{F}_d$. In particular, if $\tau$ has superexponential growth in $n$ these averages cannot be powers of a single operator, and do not have the `balanced' form considered in Theorem 3.14 of \cite{MR2573139}. It is possible that this approach could simplify some of the arguments in \cite{MR2573139}, where great care is required to deal with averages on products of Lie groups.

\subsubsection{Rich synergodicity}

Given that the property of synergodicity is nontrivial, in the sense of having both examples and nonexamples, we may ask about the extent to which arbitrary ergodic actions of a product group may be approximated by synergodic actions. Thus the following definition is natural.

\begin{definition} We define an infinite countable group $G$ to be \textbf{richly synergodic} if the synergodic actions are dense in $\mrm{Erg}(G \times G,X,\mu)$ for every standard probability space $(X,\mu)$. \end{definition}

With the notation of the above definition, it is straightforward to verify that the group $G$ is richly synergodic if and only if the synergodic actions are dense in $\mrm{Act}(G \times G,X,\mu)$ for some diffuse standard probability space. We obtain the following result, which shows that rich synergodicity is a nontrivial property of discrete groups.

\begin{theorem} \label{thm.rich} \label{thm.amenable} \begin{description} \item[(a)] Every countably infinite amenable group is richly synergodic. \item[(b)] The rank-$d$ free group $\mathbb{F}_d$ not richly synergodic when $d \geq 2$. \end{description} \end{theorem}

 The proof of Clause (a) in the above theorem is fairly straightforward, but the proof of Clause (b) relies on the negation of Connes' embedding conjecture which was recently announced in \cite{10.1145/3485628}. The technique we use to prove Clause (b) can be directly extended to establish that any group containing a nonabelian free subgroup fails to be richly synergodic. We consider following conjecture to be highly plausible, although establishing it for nonamenable groups without free subgroups would likely require the somewhat complex coinduction-from-subequivalence-relation methods used in \cite{MR2529949}. 

\begin{conjecture} A countably infinite group is richly synergodic if and only if it is amenable. \end{conjecture}

\subsection{Acknowledgments}

We thank Konstantin Slutsky for several helpful suggestions that improved the article.

\section{Averaging on product groups} \label{sec.avv}

\subsection{A weighted ergodic theorem} \label{sec.avvw}

In Section \ref{sec.avv} we deal with some initial considerations about ergodic averages on product groups. We begin with the present Subsection \ref{sec.avvw}, which formulates certain definitions required to cite the statement of a weighted ergodic theorem for actions of discrete groups. If $S$ is a countable set and $\mathfrak{w}$ is a measure on $S$, for $s \in S$ we will typically write $\mathfrak{w}(s)$ instead of $\mathfrak{w}(\{s\})$. If $G$ is a countable group and $\mathfrak{w}$ and $\mathfrak{u}$ are probability measures on $G$, we define the \textbf{convolution product} to be the probability measure $\mathfrak{w} \ast \mathfrak{u}$ on $G$ given by \[ [\mathfrak{w} \ast \mathfrak{u}](g) = [\mathfrak{w} \times \mathfrak{u}]\Bigl(\bigl\{ (h,k) \in G \times G: hk = g \bigr \} \Bigr) \] where $\mathfrak{w} \times \mathfrak{u}$ denotes the product measure on $G \times G$. For $n \in \mathbb{N}$ we also define the $n$\textbf{-fold autoconvolution of }$\mathfrak{w}$ to be the probability measure $\mathfrak{w}^{\ast n}$ on $G$ constructed recursively as follows. \[ \mathfrak{w}^{\ast n}(g) =  \begin{cases} \mathfrak{w} &\mbox{ if }n=1 \\ \mathfrak{w} \ast \left(\mathfrak{w}^{\ast(n-1)}\right) & \mbox{ if }n \geq 2 \end{cases} \] Writing $\mathfrak{w}^{\times n}$ for the measure on $G^n$ given by the $n$-fold product of $\mathfrak{w}$, we note that $\mathfrak{w}^{\ast n}$ has an alternative expression as: \begin{equation} \label{eq.conv-def} \mathfrak{w}^{\ast n}(g) = \mathfrak{w}^{\times n}\Bigl(\bigl\{ (g_1,\ldots,g_n) \in G^n: g_n \cdots g_1 = g \bigr\} \Bigr) \end{equation}   We introduce the following terminology regarding a probability measure $\mathfrak{w}$ on a countable group $G$. \begin{itemize} \item We define $\mathfrak{w}$ to be \textbf{symmetric} if $\mathfrak{w}(g) = \mathfrak{w}(g^{-1})$ for all $g \in G$. \item We define $\mathfrak{w}$ to be \textbf{generating} if the support $\{g \in G:\mathfrak{w}(g) > 0 \}$ generates the group $G$. \item We define $\mathfrak{w}$ to be \textbf{absolutely generating} if $\mathfrak{w}$ is symmetric and the measure $\mathfrak{w}^{\ast 2}$ is generating.\end{itemize}

\begin{remark} We remark that is it possible for $\mathfrak{w}$ to be symmetric and generating but not absolutely generating. For example, let $\mathfrak{d}_g$ denote the pure point measure at $g \in G$. Choose $G = \mathbb{Z}$ and $\mathfrak{w} = \frac{1}{2}(\mathfrak{d}_{-1}+\mathfrak{d}_1)$. Then $\mathfrak{w}$ is symmetric and generating but the support of the measure \[ \mathfrak{w}^{\ast 2} = \frac{1}{4}\mathfrak{d}_{-2} + \frac{1}{2} \mathfrak{d}_0 + \frac{1}{4} \mathfrak{d}_2 \] generates the proper subgroup $2\mathbb{Z}$. However, writing $e$ for the identity of $G$ it is straightforward to verify that if $\mathfrak{w}(e) > 0$ then $\mathfrak{w}$ is absolutely generating if and only if $\mathfrak{w}$ is symmetric and generating. In particular, this implies that every countable group admits an absolutely generating probability measure. \end{remark}

The following appears as Theorem 1 in \cite{MR2923460}.

\begin{theorem}[Weighted mean ergodic theorem for group actions] \label{thm.oscc} Let $G$ be a countable group and let $\mathfrak{w}$ be an absolutely generating probability measure on $G$. Then the sequence $(\mathfrak{w}^{\ast n})_{n \in \mathbb{N}}$ is $p$-mean ergodic for every $p \in [1,\infty)$.  \end{theorem}

\subsection{Convolutions of measures on product groups}

We now establish the following 

\begin{proposition} \label{prop.ind-conv} Let $G$ and $H$ be countable groups and let $\mathfrak{w}$ and $\mathfrak{u}$ be probability measures on $G$ and $H$ respectively. Then for all $n \in \mathbb{N}$ we have $(\mathfrak{w} \times \mathfrak{u})^{\ast n} = \mathfrak{w}^{\ast n} \times \mathfrak{u}^{\ast n}$.  \end{proposition}

\begin{proof}[Proof of Proposition \ref{prop.ind-conv}] In the present proof we will need to manipulate lists of tuples. In the interest of notational clarity we will delimit such a list with square brackets, reserving parentheses for the tuples which comprise the list. Fix $G,H,\mathfrak{w},\mathfrak{u}$ and $n$ as in the statement of Proposition \ref{prop.ind-conv} and define the `currying' bijection $\mathscr{C}:(G \times H)^n \to G^n \times H^n$ as follows. \[ \mathscr{C}\bigl( [(g_1,h_1),\ldots,(g_n,h_n) \bigr] \bigr) = \bigl[ (g_1,\ldots,g_n),(h_1,\ldots,h_n)\bigr]  \] If $E_1,\ldots,E_n$ are subsets of $G$ and $F_1,\ldots,F_n$ are subsets of $H$ we find that \[ \mathscr{C}^{-1}\bigl( (E_1 \times \cdots \times E_n) \times (F_1 \times \cdots \times F_n)\bigr) = (E_1 \times F_1) \times \cdots \times (E_n \times F_n) \] and therefore: \begin{equation} \label{eq.push} \mathscr{C}_{\#}\bigl( (\mathfrak{w} \times \mathfrak{u})^{\times n}  \bigr) =  (\mathfrak{w}^{\times n}) \times (\mathfrak{u}^{\times n})  \end{equation}  Now, for $g \in G$ and $h \in H$ we make the following definitions. \[ E_g  = \{(g_1,\ldots,g_n) \in G^n: g_n \cdots g_1 = g\}  \qquad  \qquad F_h = \{(h_1,\ldots,h_n) \in H^n: h_n \cdots h_1 = h\} \] Also define \[  D_{g,h} = \Bigl\{[ (g_1,h_1),\ldots,(g_n,h_n)] \in (G \times H)^n : (g_n,h_n) \cdots (g_1,h_1) = (g,h)  \Bigr \} \] Combining the last two displays with the definition (\ref{eq.conv-def}) we obtain: \[  \mathfrak{w}^{\ast n}(g) = \mathfrak{w}^{\times n}(E_g) \qquad \qquad \mathfrak{u}^{\ast n}(h) = \mathfrak{u}^{\times n}(F_h) \qquad \qquad (\mathfrak{w} \times \mathfrak{u})^{\ast n}(g,h) = (\mathfrak{w} \times \mathfrak{u})^{\times n}(D_{g,h}) \]  By construction we have $\mathscr{C}(D_{g,h}) = E_g \times F_h$ and so using (\ref{eq.push}) it follows that \[ (\mathfrak{w} \times \mathfrak{u})^{\times n}(D_{g,h}) = \mathfrak{w}^{\times n}(E_g)\,  \mathfrak{u}^{\times n}(F_h) \] and by combining the previous two displays we obtain: \[ (\mathfrak{w} \times \mathfrak{u})^{\ast n}(g,h)= \mathfrak{w}^{\ast n}(g) \mathfrak{u}^{\ast n}(h) \] This completes the proof of Proposition \ref{prop.ind-conv}. \end{proof}

From Theorem \ref{thm.oscc} and Proposition \ref{prop.ind-conv} we obtain the following. 

\begin{corollary} \label{cor.oscc} Let $G$ and $H$ be countable groups, and let $\mathfrak{w}$ and $\mathfrak{u}$ be absolutely generating probability measures on $G$ and $H$ respectively. Then the sequence $(\mathfrak{w}^{\ast n} \times \mathfrak{u}^{\ast n})_{n \in \mathbb{N}}$ is $p$-mean ergodic for all $p \in [1,\infty)$.\end{corollary}

\subsection{Ergodic averages for semi-invariant sets}

Given a subset $R$ of an ambient set $S$, we write $\mathbf{1}_R:S \to \{0,1\}$ for the indicator function of $R$. When there is an ambient probability space $(X,\mu)$ understood, we will abuse notation by identifying $z \in \mathbb{C}$ with the constant function on $X$ taking value $z$.

\begin{proposition} \label{prop.onside} Let $G$ and $H$ be countable groups and let $\mathfrak{w}$ be an absolutely generating probability measure on $G$. Let $(X,\mu)$ be a standard probability space and let $(\sf{A},\sf{B}) \in \mrm{Erg}(G \times H,X,\mu)$. Also let $D \in \eusc{E}_{\sf{B}}$. Then for any $p \in [1,\infty)$ we have: \[  \lim_{n \to \infty} ||\mu(D) -  \mathsf{A}(\mathfrak{w}^{\ast n})\mathbf{1}_D||_p = 0  \] \end{proposition}

\begin{proof}[Proof of Proposition \ref{prop.onside}] Choose an absolutely generating probability measure $\mathfrak{u}$ on $H$ and note that the probability measure $\mathfrak{w} \times \mathfrak{u}$ is an absolutely generating probability measure on $G \times H$. Since the action $(\mathsf{A},\mathsf{B})$ was assumed to be ergodic, by applying Corollary \ref{cor.oscc} we find:  \begin{equation} \label{eq.ghh-erg} \lim_{n \to \infty} ||\mu(D) -\mathsf{B}(\mathfrak{u}^{\ast n})\mathsf{A}(\mathfrak{w}^{\ast n}) \mathbf{1}_D||_p   = 0 \end{equation} Now, for all $x \in X$ we have: \begin{align} \label{eq.erg-0.9} [\mathsf{B}(\mathfrak{u}^{\ast n})\mathsf{A}(\mathfrak{w}^{\ast n}) \mathbf{1}_D](x) & = \sum_{(g,h) \in G \times H} \mathfrak{w}^{\ast n}(g) \mathfrak{u}^{\ast n}(h)\mathbf{1}_D(\mathsf{B}^h\mathsf{A}^g x) \\& =  \sum_{(g,h) \in G \times H} \mathfrak{w}^{\ast n}(g) \mathfrak{u}^{\ast n}(h)\mathbf{1}_D(\mathsf{A}^g x) \label{eq.erg-1}  \\ & =  \left( \sum_{g \in G} \mathfrak{w}^{\ast n}(g) \mathbf{1}_D(\mathsf{A}^g x)  \right) \left( \sum_{h \in H} \mathfrak{u}^{\ast n}(h) \right) \label{eq.erg-3} \\ & =  \sum_{g \in G} \mathfrak{w}^{\ast n}(g) \mathbf{1}_D(\mathsf{A}^g x)   \label{eq.erg-4} \\ & =[\mathsf{A}(\mathfrak{w}^{\ast n})\mathbf{1}_D](x) \label{eq.erg-4.1} \end{align} This computation may be justified as follows.

\begin{itemize} \item The equality in (\ref{eq.erg-0.9}) follows from the definition (\ref{eq.defop}).  \item (\ref{eq.erg-1}) follows from (\ref{eq.erg-0.9}) since $D$ was assumed to be a $\mathsf{B}$-invariant set and therefore $\mathbf{1}_D(x) = \mathbf{1}_D(\mathsf{B}^hx)$ for all $x \in X$. \item (\ref{eq.erg-3}) follows from (\ref{eq.erg-1}) by the discrete Fubini theorem since all relevant sums are absolutely convergent. \item (\ref{eq.erg-4}) follows from (\ref{eq.erg-3}) since $\mathfrak{u}^{\ast n}$ is a probability measure on $H$. \item (\ref{eq.erg-4.1}) follows from (\ref{eq.erg-4}) again by (\ref{eq.defop}). \end{itemize} Proposition \ref{prop.onside} now follows by combining (\ref{eq.ghh-erg}) with the last computation.  \end{proof}

\section{Proof of Theorems \ref{thm.synerg} and \ref{thm.locprod}}

\subsection{Proof of Theorem \ref{thm.synerg}}

\subsubsection{Reduction to Lemma \ref{lem.semi}} \label{sec.semred}

\begin{definition} \label{def.statindep} Let $(X,\mu)$ be a standard probability space and let $\eusc{S}$ and $\eusc{T}$ be two measurable $\sigma$-algebras on $X$. We define $\eusc{S}$ and $\eusc{T}$ to be \textbf{statistically independent} if $\mu(C \cap D) = \mu(C)\mu(D)$ for all $C \in \eusc{S}$ and $D \in \eusc{T}$. \end{definition}

The following lemma will be proved in Subsection \ref{sec.semprf} below.

 \begin{lemma} \label{lem.semi} Let $G$ and $H$ be countable groups, let $(X,\mu)$ be a standard probability space and let $(\sf{A},\sf{B}) \in \mrm{Erg}(G \times H,X,\mu)$. Then the $\sigma$-algebras $\eusc{E}_\sf{A}$ and $\eusc{E}_{\sf{B}}$ are statistically independent. \end{lemma}
 
In the remainder of Subsection \ref{sec.semred} we reduce the statement of Theorem \ref{thm.synerg} to Lemma \ref{lem.semi}. We begin this task by reviewing a `joining by common extension' construction which appears in Section 6.3 of \cite{MR1958753}. Let $G$ be a countable group, let $(X,\mu)$ be a standard probability space and let $\sf{A} \in \mrm{Act}(G,X,\mu)$. Let $\eusc{S}_1$ and $\eusc{S}_2$ be two $\sf{A}$-stable measurable $\sigma$-algebras on $X$. Then the factor of $\sf{A}$ corresponding to $\sigma(\eusc{S}_1,\eusc{S}_2)$ may be realized as a joining of the factors corresponding to $\eusc{S}_1$ and $\eusc{S}_2$ as follows. \\
\\
For $j \in \{1,2\}$ let $(Y_j,\nu_j)$ be a standard probability space and let $\sf{A}_j \in \mrm{Act}(G,Y_j,\nu_j)$ be a realization of $\eusc{S}_j$ with factor map $\Phi_j:X \to Y_j$. Writing $\Pi_j: Y_1 \times Y_2 \to Y_j$ for the Cartesian projection onto the $j^{\mrm{th}}$ factor, there exists a probability measure $\omega$ on $Y_1 \times Y_2$ and a map $\Phi: X \to Y_1 \times Y_2$ such that the following hold.

\begin{itemize} \item We have $(\Pi_j)_{\#} \omega = \nu_j$ for each $j \in \{1,2\}$. \item The measure $\omega$ is preserved by the diagonal action $\sf{A}_1 \parr \sf{A}_2$ of $G$ on $Y_1 \times Y_2$, and this diagonal action is a realization of the factor of $\sf{A}$ associated with $\sigma(\eusc{S}_1,\eusc{S}_2)$. \item The map $\Phi$ factors $\sf{A}$ onto $\sf{A}_1 \parr  \sf{A}_2$ in such a way that the following diagram commutes. \begin{equation} \label{eq.diagramm} \begin{tikzcd} && &X \arrow[dd, "\Phi" description] \arrow[dddrr, "\Phi_2" description] \arrow[dddll, "\Phi_1" description] &&& \\  \\ &&& Y_1 \times Y_2 \arrow[drr,"\Pi_2" description] \arrow[dll,"\Pi_1" description] &&&  \\  & Y_1 &&&& Y_2 & \end{tikzcd} \end{equation}  \end{itemize}

Now, consider two countable groups $H$ and $K$ along with an action $(\sf{B},\sf{C}) \in \mrm{Act}(H \times K,X,\mu)$. We take $G = H \times K$ and $\sf{A} = (\sf{B},\sf{C})$ throughout the above discussion. Let $\eusc{S}_1 = \eusc{E}_\sf{B}$ and $\eusc{S}_2 = \eusc{E}_\sf{C}$. Then we may assume $(Y_1,\nu_1) = (E_\sf{B},\eta_\sf{B})$ and: \[ \sf{A}_1 = (\sf{I}_{G,E_\sf{B}}, \sf{E}_{\sf{C} \curvearrowright \sf{B}}) \in \mrm{Act}(G \times H, E_\sf{B},\eta_\sf{B}) \] Similarly we may assume $(Y_2,\nu_2) = (E_\sf{C},\eta_\sf{C})$ and: \[ \sf{A}_2 = (\sf{E}_{\sf{B} \curvearrowright \sf{C}}, \sf{I}_{H,E_\sf{C}} ) \in \mrm{Act}(G \times H, E_\sf{C},\eta_\sf{C}) \] 

Taking into account Remark \ref{rem.loc}, we see that in order to prove Theorem \ref{thm.synerg} it suffices to show $\omega = \nu_1 \times \nu_2$. By commutativity of the diagram in (\ref{eq.diagramm}), for this it suffices to show that if $D_j \in \eusc{M}_{Y_j}$ for $j \in \{1,2\}$ then we have:  \[ \mu\left( \Phi_1^{-1}(D_1) \cap \Phi_2^{-1}(D_2) \right) = \mu(D_1) \mu(D_2) \] Since we chose $\eusc{S}_1 = \eusc{E}_\mathsf{B}$ and $\eusc{S}_2 = \eusc{E}_\mathsf{C}$, resetting variables by \[ H \mapsto G \qqquad K \mapsto H \qqquad \mathsf{B} \mapsto \mathsf{A} \qqquad \mathsf{C} \mapsto \mathsf{B}  \] we have successfully reduced the proof of Theorem \ref{thm.synerg} to Lemma \ref{lem.semi}.

\subsubsection{Proof of Lemma \ref{lem.semi}} \label{sec.semprf} 

Toward a proof of Lemma \ref{lem.semi}, fix countable groups $G$ and $H$ and a standard probability space $(X,\mu)$ along with an ergodic action $(\mathsf{A},\sf{B}) \in \mrm{Erg}(G \times H,X,\mu)$. Let $C \in \eusc{E}_\mathsf{A}$ and $D \in \eusc{E}_\mathsf{B}$, so we must show: \begin{equation} \mu(C \cap D) = \mu(C)\mu(D) \label{eq.want} \end{equation} For all $x \in X$, all $g \in G$ and all $h \in H$ we have: \begin{equation} \label{eq.split} \mathbf{1}_{C \cap D}(\mathsf{A}^g \mathsf{B}^hx)  = \mathbf{1}_C(\mathsf{A}^g \mathsf{B}^hx) \mathbf{1}_D(\mathsf{B}^h \mathsf{A}^g  x) = \mathbf{1}_C(\mathsf{B}^hx) \mathbf{1}_D(\mathsf{A}^g x)   \end{equation}  Here, the left equality holds since $\sf{A}^g$ and $\sf{B}^h$ commute while the right equality holds since $C$ and $D$ were assumed to be respectively $\mathsf{A}$ and $\mathsf{B}$ invariant. Now, let $\mathfrak{w}$ and $\mathfrak{u}$ be absolutely generating probability measures on $G$ and $H$ respectively. By applying Corollary \ref{cor.oscc} with $p = 1$ we find the limit below holds.  \begin{equation} \label{eq.gh-erg} \lim_{n \to \infty} ||\mu(C \cap D) - \mathsf{A}(\mathfrak{w}^{\ast n}) \mathsf{B}(\mathfrak{u}^{\ast n}) \mathbf{1}_{C \cap D}||_1 = 0 \end{equation}

On the other hand, for any $x \in X$ we may compute: \begin{align} [\mathsf{A}(\mathfrak{w}^{\ast n}) \mathsf{B}(\mathfrak{u}^{\ast n}) \mathbf{1}_{C \cap D}](x) & =  \sum_{(g,h) \in G \times H}  \mathfrak{w}^{\ast n}(g) \mathfrak{u}^{\ast n}(h)  \mathbf{1}_{C \cap D}(\mathsf{A}^g \mathsf{B}^hx) \label{eq.erg-5.9} \\ & = \sum_{(g,h) \in G \times H}  \mathfrak{w}^{\ast n}(g) \mathfrak{u}^{\ast n}(h)  \mathbf{1}_C(\mathsf{B}^hx) \mathbf{1}_D(\mathsf{A}^g x)  \label{eq.erg-6} \\   & = \Biggl(\sum_{g \in G}  \mathfrak{w}^{\ast n}(g) \mathbf{1}_D(\mathsf{A}^g x)  \Biggr)  \left( \sum_{h \in H} \mathfrak{u}^{\ast n}(h)  \mathbf{1}_C(\mathsf{B}^hx) \right) \label{eq.erg-7} \\ & = [\mathsf{A}(\mathfrak{w}^{\ast n}) \mathbf{1}_D](x) \cdot  [\mathsf{B}(\mathfrak{u}^{\ast n}) \mathbf{1}_C](x) \nonumber \end{align} Here, (\ref{eq.erg-6}) follows from (\ref{eq.erg-5.9}) by (\ref{eq.split}), while (\ref{eq.erg-7}) follows from (\ref{eq.erg-6}) by the discrete Fubini theorem and absolute convergence. Combining (\ref{eq.gh-erg}) with the previous display we find that the limit below holds. \begin{equation} \label{eq.ooo-9} \lim_{n \to \infty} ||\mu(C \cap D) - ([\mathsf{A}(\mathfrak{w}^{\ast n}) \mathbf{1}_D] \cdot  [\mathsf{B}(\mathfrak{u}^{\ast n}) \mathbf{1}_C]) ||_1 = 0  \end{equation} Since $C$ was assumed to be $\mathsf{A}$-invariant and $D$ was assumed to be $\mathsf{B}$-invariant, Proposition \ref{prop.onside} implies that the both of the following limits hold. \begin{equation} \lim_{n \to \infty} ||\mu(D) -  \mathsf{A}(\mathfrak{w}^{\ast n})\mathbf{1}_D||_2  = 0  \qqquad \lim_{n \to \infty} ||\mu(C) -  \mathsf{B}(\mathfrak{w}^{\ast n})\mathbf{1}_C||_1  = 0   \label{eq.ooo-6} \end{equation} We compute: \begin{align}  ||(\mu(D) - \mathsf{A}(\mathfrak{w}^{\ast n}) \mathbf{1}_D ) \cdot  ( \mathsf{B}(\mathfrak{u}^{\ast n})  \mathbf{1}_C)||_1   & \leq  ||\mu(D) - \mathsf{A}(\mathfrak{w}^{\ast n}) \mathbf{1}_D ||_2||\mathsf{B}(\mathfrak{u}^{\ast n})  \mathbf{1}_C||_2 \label{eq.ooo-2}  \\ & \leq ||\mu(D) - \mathsf{A}(\mathfrak{w}^{\ast n}) \mathbf{1}_D ||_2 \label{eq.ooo-3}   \end{align} Here, the inequality in (\ref{eq.ooo-2}) is Cauchy-Schwartz, while (\ref{eq.ooo-3}) follows from (\ref{eq.ooo-2}) since the operator $\mathsf{B}(\mathfrak{u}^{\ast n})$ contracts the $2$-norm and we have $||\mathbf{1}_C||_2 = \sqrt{\mu(C)} \leq 1$. We compute again: \begin{align*}  || \mu(C)\mu(D) -([\mathsf{A}(\mathfrak{w}^{\ast n}) \mathbf{1}_D] \cdot  [\mathsf{B}(\mathfrak{u}^{\ast n}) \mathbf{1}_C]) ||_1  &  \leq \mu(D) ||\mu(C) -  \mathsf{B}(\mathfrak{u}^{\ast n})\mathbf{1}_D||_1 \\ & \hspace{1 cm} + ||(\mu(D) - \mathsf{A}(\mathfrak{w}^{\ast n}) \mathbf{1}_D ) \cdot  ( \mathsf{B}(\mathfrak{u}^{\ast n})  \mathbf{1}_C)||_1  \\ &  \leq ||\mu(C) -  \mathsf{B}(\mathfrak{u}^{\ast n})\mathbf{1}_D||_1 + ||\mu(D) - \mathsf{A}(\mathfrak{w}^{\ast n}) \mathbf{1}_D ||_2   \end{align*} Here, the last equality follows from (\ref{eq.ooo-3}). Combining the previous display with (\ref{eq.ooo-6}), we find \[ \lim_{n \to \infty}  || \mu(C)\mu(D) -([\mathsf{A}(\mathfrak{w}^{\ast n}) \mathbf{1}_D] \cdot  [\mathsf{B}(\mathfrak{u}^{\ast n}) \mathbf{1}_C]) ||_1 = 0 \] The desired equality (\ref{eq.want}) now follows by now follows by combining the previous display with (\ref{eq.ooo-9}). This completes the proof of Lemma \ref{lem.semi}. 

\subsection{Proof of Theorem \ref{thm.locprod}}

Let $G$ and $H$ be countable groups, let $(X,\mu)$ and $(Y,\nu)$ be standard probability spaces and let $\sf{A} \in \mrm{Erg}(G,X,\mu)$ and $\sf{B} \in \mrm{Erg}(H,Y,\nu)$ be ergodic actions. Consider the local product $\sf{A} \sqq \sf{B} \in \mrm{Act}(G \times H, X \times Y,\mu \times \nu)$. We claim that in order to prove Theorem \ref{thm.locprod} it suffices to show the final statement in its conclusion: \begin{equation} \label{eq.oo} \eusc{E}_{\sf{A}_\square} = \eusc{N}_X \times \eusc{M}_Y \qqquad \qqquad \eusc{E}_{{}_\square \sf{B}} = \eusc{M}_X \times \eusc{N}_Y \end{equation} Indeed, if we assume the above equality holds then we find \[ \eusc{E}_{\sf{A} \sqq \sf{B}} \subseteq \eusc{E}_{\sf{A}_\square} \cap \eusc{E}_{{}_\square \sf{B}} = (\eusc{N}_X \times \eusc{M}_Y) \cap (\eusc{M}_X \times \eusc{N}_Y) = \eusc{N}_{X \times Y} \] and therefore $\sf{A} \sqq \sf{B}$ is ergodic. Similarly, we have \[ \sigma(\eusc{E}_{\sf{A}_\square}, \eusc{E}_{{}_\square \sf{B}}) = \sigma\bigl(\eusc{N}_X \times \eusc{M}_Y , \eusc{M}_X \times \eusc{N}_Y \bigr) = \eusc{M}_{X \times Y} \] so the synergodic decomposition of $(\sf{A},\sf{B})$ is an isomorphism. We now turn to a proof of the left equality in (\ref{eq.oo}). It follows immediately from the construction of the local product that $\eusc{N}_X \times \eusc{M}_Y \subseteq \eusc{E}_{\sf{A}_\square}$. Thus in order to verify the left equality in (\ref{eq.oo}) it suffices to consider a nonempty set $D \in \eusc{E}_{\sf{A}_\square}$ and show there exists a set $C \in \eusc{M}_Y$ such that: \begin{equation} (\mu \times \nu)\bigl(D \triangle (X \times C)\bigr) = 0 \label{eq.fub} \end{equation} First observe that for any $(x,y) \in X \times Y$, our assumption that $D$ is $\sf{A}_\square$-invariant implies \begin{equation} \label{eq.square} (x,y) \in D \iff (\sf{A}^gx,y) \in D \end{equation} for all $g \in G$. For $y \in Y$ we adopt the notation: \[ D_y = \{x \in X: (x,y) \in D\} \] By combining this definition with (\ref{eq.square}) we find that $D_y$ is an $\sf{A}$-invariant subset of $X$ for all $y \in Y$. Since we have assumed the action $\sf{A}$ is ergodic, it follows that $\mu(D_y) \in \{0,1\}$ for all $y \in Y$. Thus we can take \[ C= \{y \in Y:\mu(D_y) =1 \} \] to satisfy (\ref{eq.fub}). This shows that $\eusc{E}_{\sf{A}_\square}= \eusc{N}_X \times \eusc{M}_Y$. A symmetrical argument shows that $\eusc{E}_{{}_\square \sf{B}} = \eusc{M}_X \times \eusc{N}_Y$ and so the proof of Theorem \ref{thm.locprod} is complete.

\section{Proof of Theorem \ref{thm.crosserg}}

\subsection{Reduction to Lemma \ref{lem.prodd}}

If $X$ is a standard probability space, we write $\mathbb{E}_X$ for the integration operator on $X$. The following lemma will be proved in Subsection \ref{sec.proddprf} below.

\begin{lemma} \label{lem.prodd} Let $G$ and $H$ be countable groups, let $(X,\mu)$ be a standard probability space and let $(\sf{A},\sf{B}) \in \mrm{Erg}(G \times H,X,\mu)$. Then we have $\mathbb{I}_\sf{B} \mathbb{I}_{\sf{A}} = \mathbb{E}_X$. \end{lemma}

We now reduce Theorem \ref{thm.crosserg} to Lemma \ref{lem.prodd}. Let $p \in [1,\infty]$, let $G$ and $H$ be countable groups, let $(X,\mu)$ be a standard probability space and let $\sf{A} \in \mrm{Act}(G,X,\mu)$. Also let $\mathsf{B} = \mathsf{I}_H^X$ be the trivial action of $H$ on $(X,\mu)$. Then for any $\psi \in L^p(X,\mu)$ and any pair $\mathfrak{w}$ and $\mathfrak{u}$ of probability measures on $G$ and $H$ respectively we have:  \[ \sum_{(g,h) \in G \times H} \mathfrak{w}(g) \mathfrak{u}(h) \psi(\mathsf{B}^h \mathsf{A}^g x) = \sum_{(g,h) \in G \times H} \mathfrak{w}(g) \mathfrak{u}(h) \psi(\mathsf{A}^g x) = \sum_{g \in G} \mathfrak{w}(g)\psi(\mathsf{A}^g x)  \] Here, the left inequality holds since the action of $\mathsf{B}^h$ is trivial, while the right inequality holds since $\mathfrak{u}$ is a probability measure. This shows that the ergodic averages whose convergence is asserted by Clause (a) of Theorem \ref{thm.crosserg} are special cases of those whose convergence is asserted by Clause (b), and so Clause (b) implies Clause (a). Thus we turn to the proof that Clause (a) implies Clause (b).\\
\\
To this end, we now let $(\sf{A},\sf{B}) \in \mrm{Erg}(G \times H,X,\mu)$ be ergodic and let $p \in [1,\infty]$. Also let $(\mathfrak{w}_n)_{n \in \mathbb{N}}$ and $(\mathfrak{u}_n)_{n \in \mathbb{N}}$ be $p$-mean ergodic sequences of probability measures on $G$ and $H$ respectively. Let $\psi \in L^p(X,\mu)$ and let $\epsilon > 0$. Using the definition of $p$-mean ergodicity for $(\mathfrak{w}_n)_{n \in \mathbb{N}}$, we see there exists $N \in \mathbb{N}$ such that for all $n \geq N$ we have: \[ ||(\sf{A}(\mathfrak{w}_n)-\mathbb{I}_\sf{A}) \psi||_p \leq \frac{\epsilon}{2}  \] For any $m \in \mathbb{N}$ the operator $\sf{B}(\mathfrak{u}_m)$ contracts the $p$-norm and so it follows from the above display that if $n \geq N$ then: \[ ||\mathsf{B}(\mathfrak{u}_m)(\sf{A}(\mathfrak{w}_n)-\mathbb{I}_\sf{A}) \psi||_p \leq \frac{\epsilon}{2} \]

Moreover, we have $\mathbb{I}_\sf{A}\psi \in L^p(X,\mu)$ and so using the definition of $p$-mean ergodicity for $(\mathfrak{u}_n)_{n \in \mathbb{N}}$ we see there  there exists $M \in \mathbb{N}$ such that for all $m \geq M$ we have: \[ ||(\sf{B}(\mathfrak{u}_m)-\mathbb{I}_\sf{B}) \mathbb{I}_\sf{A}\psi||_p \leq \frac{\epsilon}{2}  \] From the last two displays we see that if $n \geq N$ and $m \geq M$ then: \[ ||(\sf{B}(\mathfrak{u}_m)\mathsf{A}(\mathfrak{w}_n) - \mathbb{I}_\sf{B}\mathbb{I}_{\sf{A}}) \psi ||_p \leq ||(\sf{B}(\mathfrak{u}_m)-\mathbb{I}_\sf{B}) \mathbb{I}_\sf{A}\psi||_p+||\mathsf{B}(\mathfrak{u}_m)(\sf{A}(\mathfrak{w}_n)-\mathbb{I}_\sf{A}) \psi||_p \leq \epsilon \] Thus if $\ell \in \mathbb{N}$ satisfies $\ell \geq \max(N,M)$ then the previous display implies: \[ \nml ( \sf{B}(\mathfrak{u}_{\ell}) \sf{A}(\mathfrak{w}_{\ell})  - \mathbb{I}_\sf{B}\mathbb{I}_{\sf{A}}) \psi \nmr_p \leq \epsilon  \] Thus assuming Lemma \ref{lem.prodd} we have: \[ \lim_{n \to \infty} \nml ( \sf{B}(\mathfrak{u}_{\ell}) \sf{A}(\mathfrak{w}_{\ell})  - \mathbb{E}_X) \psi \nmr_p = 0\] Our assumption that $(\sf{A},\sf{B})$ is ergodic implies $\mathbb{E}_X = \mathbb{I}_{\sf{A},\sf{B}}$ and so Theorem \ref{thm.crosserg} is established for ergodic actions. The general case follows as usual by considering the ergodic decomposition of an arbitrary action $(\sf{A},\sf{B}) \in \mrm{Act}(G \times H,X,\mu)$.

\subsection{Conditional expectations on cubes} \label{sec.cube}

Subsection \ref{sec.cube} serves to prove the following elementary measure-theoretic statement.

\begin{proposition}\label{prop.stat} For $j \in \{1,2,3\}$ let $(X_j,\mu_j)$ be a standard probability space and let: \[ (Y,\nu) = (X_1 \times X_2 \times X_3,\mu_1 \times \mu_2 \times \mu_3)  \] Also let $\Pi_j: Y \to X_j$ be the Cartesian projection and let $\eusc{P}_j$ be the measurable $\sigma$-algebra on $(Y,\nu)$ given by $(\Pi_j)_{\uparrow}[\eusc{M}_{X_j}]$. Also let $\mathbb{P}_j$ be the conditional expectation from $L^1(Y,\nu)$ onto $\eusc{P}_j$. Then for any distinct pair $j,k \in \{1,2,3\}$ we have $\mathbb{P}_j \mathbb{P}_k = \mathbb{E}_Y$. \end{proposition}

\begin{proof}[Proof of Proposition \ref{prop.stat}] Without loss of generality we may assume that $j = 1$ and $k = 2$. Let $\mathcal{V}$ denote the closed subspace of $L^1(Y,\nu)$ defined as follows. \[ \mathcal{V} = \bigl \{ \psi \in L^1(Y,\nu): \psi(x_1,x_2,x_3) = \psi(y_1,y_2,y_3) \mbox{ if }x_2 = y_2 \bigr\} \] Thus $\mathcal{V}$ is the subspace of $\eusc{P}_2$-measurable functions, or equivalently $\mathcal{V}$ is the range of $\mathbb{P}_2$. Thus it suffices to verify that if $\psi \in \mathcal{V}$ then $\mathbb{P}_1 \psi = \mathbb{E}_Y \psi$.  Since the value of a function $\psi \in \mathcal{V}$ does not depend on the first coordinate, for any $(x_1,x_2,x_3) \in Y$ we have: \[ \psi(x_1,x_2,x_3) = \int_{X_1} \psi(x,x_2,x_3) \deee \mu_1(x) \] According to Item 9 in Theorem A.8 of \cite{MR1958753} the following disintegration formula holds for all $\psi \in L^1(Y,\nu)$. \[ [\mathbb{P}_1 \psi](x_1,x_2,x_3) = \int_{X_3} \int_{X_2} \psi(x_1,y,z) \deee \mu_2(y) \deee \mu_3(z) \] Combining the last two displays we find that if $\psi \in \mathcal{V}$ then \[  [\mathbb{P}_1 \psi](x_1,x_2,x_3) = \int_{X_3} \int_{X_2} \int_{X_1} \psi(x,y,z)  \deee \mu_1(x)  \deee \mu_2(y)  \deee \mu_3(z) \] as required. \end{proof}

\subsection{Rokhlin skew-product theorem} \label{sec.rokh}

In Subsection \ref{sec.rokh} we recall the `Rohklin skew-product theorem' which gives an more explicit version of Theorem \ref{thm.factor}. Let $G$ be a countable discrete group and let $\sf{A} \in \mrm{Act}(G,X,\mu)$. Given a Polish group $\Xi$, we define a \textbf{cocycle of }$\mathsf{A}$\textbf{ with values in }$\Xi$ to be a measurable map $\alpha:G \times X \to \Xi$ which satisfies the cocycle equation \[ \alpha(gh,x) = \alpha(g,hx)\alpha(h,x) \] for all $g,h \in G$ and $\mu$-almost every $x \in X$. Here, measurability of $\alpha$ is understood to mean that for every Borel subset $K$ of $\Xi$ the preimage $\alpha^{-1}(K)$ lies in the $\sigma$-algebra on $G \times X$ generated by all sets of the form $\{g\} \times D$ for $ g \in G$ and $D \in \eusc{M}_X$. If $\Xi = \mrm{Aut}(Y,\nu)$ for a standard probability space $(Y,\nu)$ we may construct the \textbf{skew-product extension of }$\sf{A}$\textbf{ by }$\alpha$, which is an element $\sf{A} \times_\alpha Y$ of $\mrm{Act}(G,X \times Y,\mu \times \nu)$ defined by letting $g \in G$ act according to the formula $(x,y) \mapsto (\sf{A}^gx, \alpha(g,x)y)$.

\begin{theorem} \label{thm.skew} Let $G$ be a countable group, let $(X,\mu)$ be a standard probability space and let $\sf{A} \in \mrm{Erg}(G,X,\mu)$ be an ergodic action. Then any factor $\sf{B} \in \mrm{Act}(G,Y,\nu)$ of $\sf{A}$ is isomorphic to a skew-product extension. More explicitly, there exists a standard probability space $(Z,\omega)$ and a measurable cocycle $\alpha: G \times Y \to \mrm{Aut}(Z,\omega)$ such that $\sf{A}$ is isomorphic to $\sf{B} \times_\alpha Z \in \mrm{Act}(G,Y \times Z,\nu \times \omega)$.\end{theorem}

\subsection{Proof of Lemma \ref{lem.prodd}} \label{sec.proddprf}

Fix countable groups $G_1$ and $G_2$ along with a standard probability space $(X,\mu)$ and an action $(\sf{A}_1,\sf{A}_2) \in \mrm{Erg}(G,X,\mu)$. For $j \in \{1,2\}$ we simplify notation by writing $\eusc{E}_j$ for the $\sf{A}_j$-invariant $\sigma$-algebra $\eusc{E}_{\sf{A}_j}$ and $\mathbb{I}_j$ for the $\sf{A}_j$-invariant expectation $\mathbb{I}_{\sf{A}_j}$. Similarly, we write $(E_j,\eta_j)$ for the space $(E_{\mathsf{A}_j},\eta_{\mathsf{A}_j})$ of ergodic components of $\mathsf{A}_j$ and write $\mathsf{E}_j$ for the action $\mathsf{E}_{\mathsf{A}_j \curvearrowright \mathsf{A}_{1-j}}$. Also write $\eusc{M}_j$ and $\eusc{N}_j$ respectively for the full and trivial measurable $\sigma$-algebras on $E_j$. With these notations, our goal is to show that $\mathbb{I}_1 \mathbb{I}_2 = \mathbb{E}_X$. \\
\\
Using Theorem \ref{thm.synerg} we obtain a realization of the synergodic decomposition of $(\sf{A}_1,\sf{A}_2)$ as the local product $\sf{E}_1 \sqq  \sf{E}_2 \in \mrm{Act}(G_1 \times G_2 , E_1 \times E_2, \eta_1 \times \eta_2)$. Letting $\Phi:X \to E_1 \times E_2$ be the associated factor map, by construction we have: \begin{equation} \label{eq.eff-2} \eusc{E}_1 = \Phi_\uparrow[\eusc{M}_1 \times \eusc{N}_2] \qqquad  \qqquad \eusc{E}_2 = \Phi_\uparrow[\eusc{N}_1 \times \eusc{M}_2] \end{equation} Using the assumed ergodicity of $(\sf{A}_1,\sf{A}_2)$, Theorem \ref{thm.skew} provides a standard probability space $(Y,\nu)$ and a measurable cocycle $\alpha: G \times (E_1 \times E_2) \to \mrm{Aut}(Y,\nu)$ such that $\sf{A}$ is isomorphic to the action: \[ (\sf{E}_1 \sqq  \sf{E}_2)  \times_\alpha Z \in \mrm{Act}(G_1 \times G_2,E_1 \times E_2 \times Y,\eta_1 \times \eta_2 \times \nu) \] Moreover, the factor map $\Phi$ is implemented as the projection $(x_1,x_2,y) \mapsto (x_1,x_2)$ for $(x_1,x_2,y) \in E_1 \times E_2 \times Y$. Combining the last statement with (\ref{eq.eff-2}) we find $\eusc{E}_1 = \eusc{M}_1 \times \eusc{N}_2 \times \eusc{N}_Z$ and $\eusc{E}_2 = \eusc{N}_1 \times \eusc{M}_2 \times \eusc{N}_Z$. Thus Lemma \ref{lem.prodd} follows from Proposition \ref{prop.stat}.

\section{Proof of Theorem \ref{thm.amenable}}

\subsection{Proof of Clause (a) in Theorem \ref{thm.amenable}}

For a countable set $S$, a function $\psi:S \to \mathbb{C}$ and $z \in \mathbb{C}$ we define the notation $\lim_{s \to \infty} \psi(s) = z$ to mean that for every $\epsilon > 0$ there exists a finite subset $F$ of $S$ such that $|\psi(s) - z| \leq \epsilon$ for all $s \in S \setminus F$. We recall that a p.m.p. action $\mathsf{A} \in \mrm{Act}(G,X,\mu)$ of a countable group $G$ is defined to be \textbf{strongly mixing} if  the following holds for all $C,D \in \eusc{M}_X$. \[ \lim_{g \to \infty} \mu ( \mathsf{A}^g C \cap D ) = \mu(C)\mu(D) \] We also have the following fact, which follows from Theorem 3.11 in \cite{MR1958753}. (In fact a related notion of `weak mixing' suffices for Proposition \ref{thm.mixing}, but we deal with strong mixing as the definition is much simpler and the stronger condition will be satisfied in the cases of interest to us.)

\begin{proposition} \label{thm.mixing} Let $G$ be a countable group  and let $(X,\mu)$ and $(Y,\nu)$ be standard probability spaces. Also let $\mathsf{A} \in \mrm{Act}(G,X,\mu)$ and $\mathsf{B} \in \mrm{Act}(G,Y,\nu)$ be such that $\mathsf{A}$ is strongly mixing and $\mathsf{B}$ is ergodic. Then the diagonal action $\mathsf{A} \parr \mathsf{B} \in \mrm{Act}(G,X \times Y,\mu \times \nu)$ is ergodic. \end{proposition}

In particular, since the diagonal action $\mathsf{A} \parr \mathsf{B}$ corresponds to the restriction of $\mathsf{A} \sqq \mathsf{B}$ to the diagonal subgroup $\{(g,g):g \in G\}$, we obtain the following.

\begin{corollary} \label{cor.mixing} Let $G$ be a countable group  and let $(X,\mu)$ and $(Y,\nu)$ be standard probability spaces. Also let $\mathsf{A} \in \mrm{Act}(G,X,\mu)$ and $\mathsf{B} \in \mrm{Act}(G,Y,\nu)$ be such that $\mathsf{A}$ is strongly mixing and $\mathsf{B}$ is ergodic. Then the local product $\mathsf{A} \sqq \mathsf{B} \in \mrm{Act}(G \times G,X \times Y,\mu \times \nu)$ is ergodic. \end{corollary}

We also recall that a \pmp action $\sf{A} \in \mrm{Act}(G,X,\mu)$ is defined to be \textbf{free} if for all nontrivial elements $g \in G$ we have: \[ \mu(\{x \in X: \mathsf{A}^gx = x \}) = 0  \] Consider two standard probability spaces $(X,\mu)$ and $(Y,\nu)$ and suppose that $\sf{A} \in \mrm{Act}(G,X,\mu)$ and $\sf{B} \in \mrm{Act}(G,Y,\nu)$ are free actions. Then for any $(g,h) \in G \times G$ we have: \[ \bigl\{(x,y) \in X \times Y: (\mathsf{A}^gx,\mathsf{B}^hy) = (x,y)\bigr\} = \bigl(\{x \in X: \mathsf{A}^gx = x \} \times Y\bigr)  \cap \bigl(X \times \{y \in Y: \mathsf{B}^hy = y \} \bigr) \] If $(g,h)$ is nontrivial as an element of $G \times G$, then at least one of the sets in the intersection on the right of the previous display has $(\mu \times \nu)$-measure zero. Thus we obtain the following.

\begin{proposition} \label{prop.free} Let $(X,\mu)$ and $(Y,\nu)$ be standard probability spaces and suppose that $\sf{A} \in \mrm{Act}(G,X,\mu)$ and $\sf{B} \in \mrm{Act}(G,Y,\nu)$ are free actions. Then the local product $\mathsf{A} \sqq \mathsf{B} \in \mrm{Act}(G \times G,X \times Y,\mu \times \nu)$ is free. \end{proposition}

Now, let $\upsilon$ denote the uniform measure on the $2$-point set $\{0,1\}$ and let $G$ be a countable group, initially assumed to be arbitrary. Also let \[ \mathsf{L}_G \sqq \mathsf{L}_G \in \mrm{Act}\left(G \times G, \{0,1\}^G \times \{0,1\}^G,\upsilon^G \times \upsilon^G \right)\] be the local product of the left Bernoulli $G$-shift over $(X,\mu)$ with itself. Since the action $\mathsf{L}_G$ is strongly mixing and free as long as $G$ is infinite, by combining Corollary \ref{cor.mixing} and Proposition \ref{prop.free} we find that $\sf{L}_G \sqq \sf{L}_G$ is free and ergodic.\\
\\
Assume now that $G$ is amenable. According to Proposition 13.2 in \cite{MR2583950}, this implies that the isomorphic copies of any free ergodic action of $G$ are dense in $\mrm{Act}(G \times G,X,\mu)$ for any diffuse standard probability space $(X,\mu)$. By applying this to $\sf{L}_G \sqq \sf{L}_G$, we find that $G$ is richly synergodic, and so the proof of Clause (a) in Theorem \ref{thm.amenable} is complete.

\subsection{Generalities on $C^\ast$-algebras} \label{sec.seestar}

In Subsection \ref{sec.seestar} we recall some basic constructions in the theory of $C^\ast$-algebras. If $\mathfrak{A}$ is a $C^\ast$-algebra, we will write $||\cdot||_\mathfrak{A}$ for the norm of $\mathfrak{A}$. If $\mathcal{H}$ is a Hilbert space, we write $\mrm{End}(\mathcal{H})$ for the $C^\ast$-algebra of all bounded operators on $\mathcal{H}$. In this case we write $||\cdot||_{\mrm{op}}$ instead of $||\cdot||_{\mrm{End}(\mathcal{H})}$ and $\mrm{Un}(\mathcal{H})$ instead of $\mrm{Un}(\mrm{End}(\mathcal{H}))$. As usual, we will refer to an $\ast$-homomorphism from $\mathfrak{A}$ to $\mrm{End}(\mathcal{H})$ as an $\ast$\textbf{-representation of }$\mathfrak{A}$\textbf{ on }$\mathcal{H}$. \\
\\
We let $\mathfrak{A} \odot \mathfrak{B}$ denote the \textbf{algebraic tensor product of }$\mathfrak{A}$\textbf{ and }$\mathfrak{B}$, which is an $\ast$-algebra without a specified norm. Given a $\ast$-algebra $\mathfrak{C}$ and two $\ast$-homomorphisms $\alpha:\mathfrak{A} \to \mathfrak{C}$ and $\beta:\mathfrak{B} \to \mathfrak{C}$ which commute in the sense that $\alpha(a)\beta(b) = \beta(b)\alpha(a)$ for all $a \in A$ and $b \in b$, we let $\alpha \odot \beta$ denote the \textbf{algebraic tensor product of }$\alpha$\textbf{ and }$\beta$, which is an $\ast$-homomorphism from $\mathfrak{A} \odot \mathfrak{B}$ to $\mathfrak{C}$ that extends $\alpha$ and $\beta$ in the natural way. We let $\mathfrak{A} \ootimes \mathfrak{B}$ denote the \textbf{maximal }$C^\ast$\textbf{-tensor product of }$\mathfrak{A}$\textbf{ and }$\mathfrak{B}$, which is the $C^\ast$-algebra given by completing $\mathfrak{A} \odot \mathfrak{B}$ in the norm defined on an element $\phi \in \mathfrak{A} \odot \mathfrak{B}$ as follows. \begin{equation} \label{eq.maximal} ||\phi||_{\mathfrak{A} \ootimes \mathfrak{B}} = \sup \Bigl\{ ||(\alpha \odot \beta)(\phi)||_{\mrm{op}}: \alpha \mbox{ and }\beta \mbox{ are commuting }\ast\mbox{-representations of }\mathfrak{A} \mbox{ and }\mathfrak{B} \Bigr \} \end{equation} We also let $\mathfrak{A} \uotimes \mathfrak{B}$ denote the \textbf{minimal }$C^\ast$\textbf{-tensor product of }$\mathfrak{A}$\textbf{ and }$\mathfrak{B}$, which is the $C^\ast$-algebra given by completing $\mathfrak{A} \odot \mathfrak{B}$ in the norm defined on an element $\phi \in \mathfrak{A} \odot \mathfrak{B}$ as follows. \begin{equation} \label{eq.minimal} ||\phi||_{\mathfrak{A} \ootimes \mathfrak{B}} = \sup \Bigl\{ ||(\alpha \otimes \beta)(\phi)||_{\mrm{op}}: \alpha \mbox{ and }\beta \mbox{ are }\ast\mbox{-representations of }\mathfrak{A} \mbox{ and }\mathfrak{B} \Bigr \} \end{equation}  We refer the reader to Chapter 3 in \cite{MR2391387} for more information on these tensor products.\\
\\
 If $G$ is a countable group we let $\mathbb{C}[G]$ denote the group ring of $G$ with complex coefficients. We enrich $\mathbb{C}[G]$ with the structure of an $\ast$-algebra by defining $\phi^\ast(g) = \ov{\phi(g^{-1})}$ for $\phi \in \mathbb{C}[G]$. For a Hilbert space $\mathcal{H}$ and $\rho \in \mrm{Rep}(G,\mathcal{H})$, using linear extension we can define an element $\rho(\phi)$ of $\mrm{End}(\mathcal{H})$ for any $\phi \in \mathbb{C}[G]$. Moreover, the map $\phi \mapsto \rho(\phi)$ is an $\ast$-homomorphism from $\mathbb{C}[G]$ to $\mrm{End}(\mathcal{H})$. We also write $C^\ast(G)$ for the full group $C^\ast$-algebra of $G$ as in Section 2.5 of \cite{MR2391387}, and note that an element of $\mathbb{C}[G \times G]$ can be naturally identified with an element of the algebraic tensor product $C^\ast(G) \odot C^\ast(G)$.

\subsection{Ergodic theory to representation theory} \label{sec.actrep}

In Subsection \ref{sec.actrep} we discuss some connections between global aspects of ergodic theory and representation theory. Given a countable group $G$ and a separable Hilbert space $\mathcal{H}$, we let $\mrm{Rep}(G,\mathcal{H})$ be the set of unitary representations of $G$ on $\mathcal{H}$. We endow the unitary group $\mrm{Un}(\mathcal{H})$ with the strong operator topology, which is defined by the stipulation that a sequence $(u_n)_{n \in \mathbb{N}}$ of elements of $\mrm{Un}(\mathcal{H})$ converges to $u \in \mrm{Un}(\mathcal{H})$ if and only if we have \[ \lim_{n \to \infty} ||(u_n - u)\vartheta||_\mathcal{H} = 0 \] for all $\vartheta \in \mathcal{H}$. We then endow $\mrm{Rep}(G,\mathcal{H})$ with the topology it inherits as a closed subspace of $\mrm{Un}(\mathcal{H})^G$. The group $\mrm{Un}(\mathcal{H})$ acts on $\mrm{Rep}(G,\mathcal{H})$ by letting $[u \cdot \rho](g) = u^{-1}\rho(g)u$ for $\rho \in \mrm{Rep}(G,\mathcal{H}), u \in \mrm{Un}(\mathcal{H})$ and $g \in G$. \\
\\
If $(X,\mu)$ is a standard probability space and $\sf{T} \in \mrm{Aut}(X,\mu)$ we write $\mathbf{k}_\mathsf{T}$ for the \textbf{Koopman operator of }$\sf{T}$, which is the element of $\mrm{Un}(L^2(X,\mu))$ defined by setting \[ [\mathbf{k}_\mathsf{T}(g)\psi](x) = \psi(\mathsf{T}^{-1}x)  \] for $\psi \in L^2(X,\mu)$ and $x \in X$. Given $\sf{A} \in \mrm{Act}(G,X,\mu)$, we also write $\mathbf{k}_\sf{A}$ for the \textbf{Koopman representation of }$\sf{A}$, which is the element of $\mrm{Rep}(G,L^2(X,\mu))$ defined by letting $\mathbf{k}_\sf{A}(g) = \mathbf{k}_{\sf{A}^g}$ for $g \in G$. The following result appears as Proposition H.14 in \cite{MR2583950}.

\begin{proposition} \label{prop.onee} For any countable group the family of representations \[ \bigl\{u \cdot \mathbf{k}_\mathsf{A}:\mathsf{A} \in \mrm{Act}(G,X,\mu), u \in \mrm{Un}(L^2(X,\mu)) \bigr\} \] is dense in $\mrm{Rep}(G,L^2(X,\mu))$.\end{proposition}

In Section (B) of Chapter 1 in \cite{MR2583950} it is shown that the map $\sf{T} \to \mathbf{k}_\mathsf{T}$ from $\mrm{Aut}(X,\mu)$ to $\mrm{Un}(L^2(X,\mu))$ is continuous. Hence the map $\mathsf{A} \mapsto \mathbf{k}_\mathsf{A}$ from $\mrm{Act}(G,X,\mu)$ to $\mathrm{Rep}(G,L^2(X,\mu))$ is continuous. Thus we we obtain the following corollary of Proposition \ref{prop.onee}.

\begin{corollary} \label{cor.three} Let $G$ be a countable group and let $S$ be a dense subset of $\mrm{Act}(G,X,\mu)$. Then we have that the family \[ \bigl\{u \cdot \mathbf{k}_\mathsf{A}:\mathsf{A} \in S, u \in \mrm{Un}(L^2(X,\mu)) \bigr\} \] is dense in $\mrm{Rep}(G,L^2(X,\mu))$. \end{corollary} 

We now show the following connection between operator norms of group ring elements and the topology of $\mrm{Rep}(G,\mathcal{H})$.

\begin{proposition} \label{prop.three} For any group $G$ and any $\phi \in \mathbb{C}[G]$ we have the the function $\rho \mapsto ||\rho(\phi)||_{\mrm{op}}$ is a lower semicontinuous map from $\mrm{Rep}(G,\mathcal{H})$ to $\mathbb{R}$. \end{proposition}

\begin{proof}[Proof of Proposition \ref{prop.three}] First observe that for any $\vartheta \in \mathcal{H}$ we have: \[ ||\rho(\phi)\vartheta||^2 = \langle \rho(\phi) \vartheta, \rho(\phi) \vartheta \rangle = \langle \rho(\phi)^\ast \rho(\phi) \vartheta, \vartheta \rangle = \langle \rho(\phi^\ast \phi) \vartheta, \vartheta \rangle \] Therefore can express the operator norm as follows: \begin{equation} ||\rho(\phi)||_{\mrm{op}}^2 = \sup \bigl \{ \langle \rho(\phi^\ast \phi) \vartheta,  \vartheta \rangle : \vartheta \in \mathcal{H} \mbox{ is a unit vector}  \bigr\} \label{eq.opnorm} \end{equation} Now, suppose that $(\rho_n)_{n \in \mathbb{N}}$ is a sequence of elements of $\mrm{Rep}(G,\mathcal{H})$ which converges to an element $\rho \in \mrm{Rep}(G,\mathcal{H})$. Fix $\phi \in \mathbb{C}[G]$ and a unit vector $\vartheta \in \mathcal{H}$.  From the definition of the topology on $\mrm{Rep}(G,\mathcal{H})$, we see that the following holds for all $g \in G$: \[ \lim_{n \to \infty} \langle \rho_n(g) \vartheta, \vartheta \rangle = \langle \rho(g)\vartheta, \vartheta \rangle \] Since $\phi$ is finitely supported, it follows that: \[ \lim_{n \to \infty} \langle \rho_n(\phi^\ast \phi) \vartheta, \vartheta \rangle = \langle \rho(\phi^\ast \phi)\vartheta, \vartheta \rangle \] By applying (\ref{eq.opnorm}) to $\rho_n$, we see that the following holds for all $n \in \mathbb{N}$. \[ ||\rho_n(\phi)||_{\mrm{op}}^2 \geq \langle \rho_n(\phi^\ast \phi) \vartheta, \vartheta \rangle \] Combining the previous two displays we obtain: \[ \liminf_{n \to \infty} ||\rho_n(\phi)||_{\mrm{op}}^2 \geq \langle \rho(\phi^\ast \phi)\vartheta, \vartheta \rangle  \] Since $\vartheta$ was an arbitrary unit vector in $\mathcal{H}$, we can combine the previous display with (\ref{eq.opnorm}) to find \[ \liminf_{n \to \infty} ||\rho_n(\phi)||_{\mrm{op}}^2 \geq ||\rho(\phi)||_{\mrm{op}}^2 \] and so the proof of Proposition \ref{prop.three} is complete. \end{proof}

It is clear that for any $\rho \in \mrm{Rep}(G,\mathcal{H})$, any $u \in \mrm{Un}(\mathcal{H})$ and any $\phi \in \mathbb{C}[G]$ we have $||[u \cdot \rho](\phi)||_{\mrm{op}} = ||\rho(\phi)||_{\mrm{op}}$. Thus from Corollary \ref{cor.three} and Proposition \ref{prop.three} we obtain the following.

\begin{corollary} \label{cor.four} Let $G$ be a countable group and let $\phi \in \mathbb{C}[G]$. Also let $S$ be a dense subset of $\mrm{Act}(G,X,\mu)$.  Then for any unitary representation $\rho$ of $G$ on a Hilbert space we have: \[ ||\rho(\phi)||_{\mrm{op}} \leq \sup\{ ||\mathbf{k}_\mathsf{A}(\phi)||_{\mrm{op}}: \mathsf{A} \in S \} \] \end{corollary}

\subsection{Rich synergodicity and tensor products}

We now establish the following result connecting synergodicity and the minimal tensor product norm.

\begin{proposition} \label{prop.cross} Let $G$ be a countable group and let $(\mathsf{A},\mathsf{B}) \in \mrm{Act}(G \times G,X,\mu)$ be a synergodic action. Then for any $\phi \in \mathbb{C}[G \times G]$ we have: \[ ||\mathbf{k}_{\mathsf{A},\mathsf{B}}(\phi) ||_{\mrm{op}} \leq ||\phi||_{\uotimes} \]  \end{proposition}

\begin{proof}[Proof of Proposition \ref{prop.cross}] According to Theorem \ref{thm.locprod}, there exist standard probability spaces $(Y,\nu)$ and $(Z,\omega)$ along with actions $\sf{C} \in \mrm{Act}(G,Y,\nu)$ and $\sf{D} \in \mrm{Act}(G,Z,\omega)$ such that $(\sf{A},\sf{B})$ is isomorphic to the local product $\sf{C} \sqq \sf{D}$. Then we have $\mathbf{k}_{\mathsf{C} \square \mathsf{D}} = \mathbf{k}_\mathsf{C} \otimes \mathsf{k}_\mathsf{D}$ where $\mathbf{k}_\mathsf{C}: G\to \mrm{Un}(L^2(Y,\nu))$ and $\mathbf{k}_\mathsf{D}:G \to \mrm{Un}(L^2(Z,\omega))$  are the individual Koopman representations and the tensor product corresponds to the Hilbert space decomposition: \[ L^2(X,\mu) \cong L^2(Y,\nu) \otimes L^2(Z,\omega) \] It follows directly from the definition (\ref{eq.minimal}) that: \[ ||(\mathbf{k}_\mathsf{C} \otimes \mathbf{k}_\mathsf{D}) (\phi)||_{\mrm{op}} \leq ||\phi||_{\uotimes} \] Thus we find $||\mathbf{k}_{\mathsf{C} \square \sf{D}} (\phi)||_{\mrm{op}} \leq ||\phi||_{\uotimes}$ and so $||\mathbf{k}_{\sf{A},\sf{B}}(\phi)||_{\mrm{op}} \leq ||\phi||_{\uotimes}$ as required. \end{proof}

From Corollary \ref{cor.four} and Proposition \ref{prop.cross} it follows that for a richly synergodic group $G$ and any $\phi \in \mathbb{C}[G \times G]$ we have $||\phi||_{\ootimes} \leq ||\phi||_{\uotimes}$. Since the reverse inequality is trivial, we obtain the following.

\begin{corollary} \label{cor.rich} Suppose $G$ is a richly synergodic group. Then we have $||\phi||_{\ootimes} = ||\phi||_{\uotimes}$ for all $\phi \in \mathbb{C}[G \times G]$. \end{corollary}

\subsection{Proof of Clause (b) in Theorem \ref{thm.amenable}} In the theory of operator algebras, the phrase `Connes' embedding conjecture' can refer to any of a proliferation of statements which have been shown to be equivalent to the original conjecture from \cite{MR0454659}. Of relevance to us is a formulation from \cite{MR1218321} that asserts Connes' embedding conjecture is equivalent to the statement that for some $d \geq 2$ we have: \[ C^\ast(\mathbb{F}_d) \ootimes C^\ast(\mathbb{F}_d) = C^\ast(\mathbb{F}_d) \uotimes C^\ast(\mathbb{F}_d)   \]  (See also Section 15 in \cite{MR3067294} for a somewhat more direct exposition.) Examining the definitions in (\ref{eq.maximal}) and (\ref{eq.minimal}), we see that the above equality of $C^\ast$-algebras is equivalent to the assertion that $||\phi||_{\ootimes} = ||\phi||_{\uotimes}$ for all $\phi \in \mathbb{C}[\mathbb{F}_d \times \mathbb{F}_d]$. Using Corollary \ref{cor.rich} we find the assertion that $\mathbb{F}_d$ is richly synergodic for some $d \geq 2$ implies Connes' embedding conjecture. Thus the refutation in \cite{10.1145/3485628} implies that $\mathbb{F}_d$ is not richly synergodic.

\bibliographystyle{plain}

\bibliography{/Users/pburton/Desktop/bib/pjburton-bibliography.bib}

\end{document}